\documentclass[a4paper,11pt]{article}
\usepackage{graphics,color}
\usepackage{amsmath}
\usepackage{amssymb}
\usepackage{mathrsfs}
\usepackage{cite}
\usepackage{verbatim}
\usepackage{float}
\usepackage{graphicx}
\usepackage{amsthm}
\usepackage{textcomp}
\usepackage{subfig}
\usepackage{enumerate}
\usepackage{hyperref}

\numberwithin{equation}{section}
\mathchardef\emptyset="001F

\newtheorem{Theorem}{Theorem}[section]
\newtheorem{Definition}[Theorem]{Definition}
\newtheorem{Proposition}[Theorem]{Proposition}
\newtheorem{Corollary}[Theorem]{Corollary}
\newtheorem{Lemma}[Theorem]{Lemma}

\newcommand{\nada}[1]{}

\newcommand{\eps}{\varepsilon}
\newcommand{\grad}{\nabla}

\newcommand{\homogeneousmap}{\varphi}
\newcommand{\indice}{j}

\newcommand{\mappa}{u}
\newcommand{\mappabuona}{v}
\newcommand{\mres}{\mathbin{\vrule height 1.6ex depth 0pt width
0.13ex\vrule height 0.13ex depth 0pt width 1.3ex}}
\newcommand{\N}{\numberset{N}} 
\newcommand{\newpar}{\alpha}
\newcommand{\numberset}{\mathbb}
\newcommand{\Om}{\Omega} 
\newcommand{\R}{\numberset{R}} 
\newcommand{\radius}{\ell}

\newcommand{\repa}{h} 
\newcommand{\rilP}{\overline P} 
\newcommand{\Rn}{\numberset{R}^n}

\newcommand{\Suno}{{\mathbb S}^1}
\newcommand{\totvarjac}{TV\!J}

\theoremstyle{definition}
\newtheorem{Remark}[Theorem]{Remark}

 \title{Relaxed area of $0$-homogeneous maps
in the strict $BV$-convergence 
}
\author{
Simone Carano\footnote{
Area of Mathematical Analysis, Modelling, and Applications,
             Scuola Internazionale Superiore di Studi Avanzati ``SISSA'',
Via Bonomea, 265 - 34136 Trieste, Italy. E-mail: scarano@sissa.it
                         }    
}

\begin{document}
\maketitle

\begin{abstract}
We compute the relaxed Cartesian area for general $0$-homogeneous map of bounded variation, with respect to the strict $BV$-convergence. In particular, we show that the relaxed area is finite for this class of maps and we provide an integral representation formula.
\end{abstract}


\vspace{2mm}

\noindent {\bf AMS (MOS) 2022 subject clas\-si\-fi\-ca\-tion:} 49J45, 49Q05, 49Q15, 49Q20, 28A75.

\section{Introduction}\label{sec:introduction}
 Let $\Omega\subset \R^2$ be a bounded open set and 
$v = (v_1,v_2):\Omega\rightarrow\R^2$ be a 
map of class $ C^1(\Omega;\R^2)$.
The graph
of $v$ is 
a Cartesian $2$-manifold in $\Omega\times\R^2\subset\R^{4}$ and its $2$-dimensional Hausdorff measure 
$\mathcal H^2$
and is given by\footnote{Clearly,
\eqref{classical_area} is  finite if $v\in W^{1,1}(\Om;\R^2)$ and $\det\nabla v\in L^1(\Om)$.}
\begin{align}\label{classical_area}
\mathcal A(v;\Omega):=
\int_\Omega\sqrt{1+|\nabla v_1|^2+|\nabla v_2|^2+(J v)^2}dx,
\end{align}
where 
\begin{align}
Jv:=\det \nabla v
=\frac{\partial v_1}{\partial x_1}\frac{\partial v_2}{\partial x_2}-\frac{\partial v_1}{\partial x_2}\frac{\partial v_2}{\partial x_1}
\end{align}
 is the Jacobian determinant of $v$.
As opposite to the case when the map is scalar-valued, the area functional $\mathcal A(\cdot\,; \Omega)$ 
is not convex, but only 
polyconvex in $\nabla v$, and its growth 
is not linear, due to the presence of ${\rm det}\grad v$. \\
It is interesting to extend the notion of area of a graph for singular maps. Following a well established tradition starting from \cite{L} and generalized in \cite{Ser} (see also \cite{DeGiorgi:92, AD}), a typical way to proceed is by relaxation: in order to gain coercivity properties in some  variational problems involving the area functional, a reasonable choice is to relax with respect to the $L^1$-convergence. In this way, we are allowing to define for every $\mappa\in L^1(\Omega;\R^2)$ 
\begin{equation}\label{relaxed L1}
\overline{\mathcal{A}}_{L^1}(\mappa;\Omega)
:=\inf\left\{\liminf_{k\rightarrow 
+\infty}\mathcal{A}(v_k;\Omega):(v_k)\subset C^1(\Omega;\R^2),\;v_k\xrightarrow{L^1} \mappa\right\}.
\end{equation}
It is not difficult to prove that if $\overline{\mathcal{A}}_{L^1}(\mappa;\Omega)<+\infty$ then $\mappa\in BV(\Om;\R^2)$, i.e. the domain of $\overline{\mathcal{A}}_{L^1}(\cdot\,;\Omega)$ is contained in $BV(\Om;\R^2)$. In truth, this inclusion holds strict and a characterization of the domain of the $L^1$-relaxed area is still missing in the literature, as opposite to the case of scalar valued maps, where the domain is $BV(\Om)$ and \eqref{relaxed L1} can be represented as an integral \cite{DalMaso:80,G}.
Moreover, the analysis of \eqref{relaxed L1} turns out to be very challenging, due to its non-local behaviour. Indeed, as conjectured in \cite{DeGiorgi:92} and proved in \cite{AD}, the set function $\overline{\mathcal{A}}_{L^1}(\mappa;\cdot)$ is, in general, not subadditive: In two fundamental examples, the authors provide the existence of a map $\mappa\in BV_{loc}(\R^2;\R^2)$ and of three open sets $\Om_1,\Om_2,\Om_3\subset\R^2$ such that $\Om_3\subset\Om_1\cup\Om_2$ and $\overline{\mathcal{A}}_{L^1}(\mappa;\Omega_3)>\overline{\mathcal{A}}_{L^1}(\mappa;\Omega_1)+\overline{\mathcal{A}}_{L^1}(\mappa;\Omega_2)$. In the first example, denoting by $B_\radius$ the disk of $\R^2$ centered at $0$ and of radius $\radius$,  they consider the symmetric triple point map $u_T:B_\radius\to\{\alpha,\beta,\gamma\}\subset\R^2$, which sends three identical circular sectors of $B_\radius$ to the vertices $\alpha,\beta,\gamma$ of an equilateral triangle. In the second one, they show that the non-subaddivity arises even among Sobolev maps, like the vortex function $u_V:B_\radius\to\Suno\subset\R^2$ defined by $u(x)=\frac{x}{|x|}$, $x\neq 0$. For an explicit computation of the value of $\overline{\mathcal{A}}_{L^1}(u_T;\Omega)$ and $\overline{\mathcal{A}}_{L^1}(u_V;\Omega)$ we refer to \cite{BP,S} and \cite{BES} (see also \cite{BMS}), respectively. Moreover, for the analysis of the triple point map without symmetry assumptions, we refer to \cite{BEPS}. \\
Although the $L^1$-topology induces some useful properties in Calculus of Variations, the previous examples show that we cannot avoid non-locality issues. An alternative approach is to choose a different topology in the relaxation, stronger than the $L^1$-topology, to put on the space $BV(\Om;\R^2)$ in order to possibly get rid of non-local phenomena. Following \cite{BCS,BCS2,Mc2}, we choose the strict $BV$-convergence. We recall that for $u_k,u\in BV(\Om;\R^2)$, we say that $u_k\rightarrow u$ strictly $BV(\Om;\R^2)$ if $u_k\rightarrow u$ in $L^1(\Om;\R^2)$ and $|Du_k|(\Om)\to|Du|(\Om)$, where $|\mu|(\Om)$ stands for the total variation of a Radon measure $\mu$ on $\Om$.   So we are led to investigate for every $\mappa\in BV(\Om;\R^2)$
\begin{equation}\label{relaxed BV}
\overline{\mathcal{A}}_{BV}(\mappa;\Omega)
:=\inf\left\{\liminf_{k\rightarrow 
+\infty}\mathcal{A}(v_k;\Omega):(v_k)\subset C^1(\Omega;\R^2),\;v_k\rightarrow \mappa \mbox{ strictly }BV(\Om;\R^2)\right\}.
\end{equation}
Another important functional, highly related to the area, is the total variation of the Jacobian determinant, which is classically defined for every $v\in C^1(\Om;\R^2)$ by $\totvarjac(v;\Om):=\int_{\Om}|Jv|dx$ and extended to every $u\in BV(\Om;\R^2)$ by relaxation 
$$\overline{\totvarjac}
_{BV}(u;\Om):=\inf\left\{\liminf_{k\rightarrow+\infty}
\totvarjac
(\mappabuona_k;\Om):
(\mappabuona_k)\subset C^1(\Om;\R^2),\; 
\mappabuona_k\rightarrow \mappa \text{ strictly }BV(\Om;\R^2)\right\}.$$
We refer to \cite{FFM, Mc1, Pa, DP, BePaTe:15,BePaTe:16} for weak notion of area, total variation of the Jacobian determinant and related energies via relaxation with other different topologies. Moreover, we address to \cite{GMS1,GMS2,GMS3} for an approach to the study of graph of singular maps via Cartesian Currents.\\

In the present paper, we generalize at once the results in \cite{BCS} about vortex-type maps and in \cite{BCS2} about piecewise constant $0$-homogeneous maps, by considering general $0$-homogeneous maps in $BV(B_\radius;\R^2)$.
\begin{Definition}\label{hom map}
A map $u\in BV(B_\radius;\R^2)$ is $0$-homogeneous if it is of the form
\begin{align}\label{homog map def}
u(x)=\gamma\left(\frac{x}{|x|}\right)\quad\mbox{ a.e. } x\in B_\radius
\end{align}
for some $\gamma\in BV(\Suno;\R^2)$.
In this case, we say that $u$ is the $0$-homogeneous (or simply homogeneous) extension of $\gamma$ on $B_\radius$.
\end{Definition}
In order to ensure the consistency of Definition \ref{hom map}, we shall prove in Proposition \ref{homog extension is BV} that the homogeneous extension of a map $\gamma\in BV(\Suno;\R^2)$ belongs to $BV(B_\radius;\R^2)$. Notice that, according to Definition \ref{hom map}, the maps $u_V$ and $u_T$ are $0$-homogeneous.\\
The fundamental idea in our analysis is to define a notion of area enclosed by the image of $\gamma$, in such a way it is compatible with the strict convergence. Precisely, we consider (compare \cite{BCS2}) the relaxation
\begin{align}\label{Plateau_rel}
	\rilP(\gamma)
:=\inf\left\{\liminf_{n\rightarrow +\infty}P(\varphi_n): 
\varphi_n\in \mathrm{Lip}(\Suno;\R^2), \varphi_n\rightarrow\gamma\text{ strictly }BV(\Suno;\R^2)
\right\}
\end{align}
of the (singular) Plateau problem 
\begin{align}\label{Plateau}
P(\varphi)=\inf\left\{\int_{B_1}|Jv|~dx: \,v\in \mathrm{Lip}(B_1;\R^2), v_{|\partial B_1}=
\homogeneousmap
\right\}
\end{align}
associated to any $\varphi\in$ Lip$(\Suno;\R^2)$.
Our main result is the following:
\begin{Theorem}\label{relaxed area thm}
Let $\gamma\in BV(\Suno;\R^2)$ and $u$ as in Definition \ref{hom map}. Then
\begin{align}\label{area BV}
\overline{\mathcal{A}}_{BV}(u;B_\radius)=\int_{B_\radius}\sqrt{1+|\nabla u|^2}dx+|D^su|(B_\radius)+\rilP(\gamma),
\end{align}
where $D^su$ is the singular part of the measure $Du$.
\end{Theorem}
A crucial ingredient in the proof of Theorem \ref{relaxed area thm} will be the computation of $\overline{{\totvarjac}}_{BV}(u,B_\radius)$ in terms of the relaxed Plateau problem \eqref{Plateau_rel} (Theorem \ref{Plateau for BV}).
In Lemma \ref{lem_plateaurel} we shall see that $\rilP(\gamma)$ can be characterized as the area enclosed by the "completed map" $\widetilde\gamma$ which "fill the jumps" of $\gamma$ by means of segments, in other words $\rilP(\gamma)=P(\widetilde\gamma)$. A precise construction of $\widetilde\gamma$ will be given in Lemma \ref{strict implies uniform 3}.\\
We point out that $\overline{\mathcal{A}}_{BV}(u;\cdot)$ is a measure for $u$ as in \eqref{homog map def}. However, to the best of our knowledge, it is not known if $\overline{\mathcal{A}}_{BV}(u;\cdot)$ is subadditive for a generic map $u\in BV(B_\radius;\R^2)$. Moreover, a complete characterization of the set Dom$(\overline{\mathcal{A}}_{BV}(\cdot\,;B_\radius)):=\{u\in BV(B_\radius;\R^2): \overline{\mathcal{A}}_{BV}(u;B_\radius)<+\infty\}$ is not yet available: from \cite{BCS2}, we only know that Dom$(\overline{\mathcal{A}}_{BV}(\cdot\,;B_\radius))\subsetneq$ Dom$(\overline{\mathcal{A}}_{L^1}(\cdot\,;B_\radius))\subsetneq BV(B_\radius;\R^2)$.

\medskip

\section{Preliminaries}\label{prelim}

Let $\Om\subset\R^2$ be an open bounded set. For any $u\in BV(\Omega;\R^2)$, we recall that
the distributional derivative $Du$ is 
a finite Radon measure valued in $\R^{2\times2}$. Denoting by $\mathscr{L}^2$ the Lebesgue measure of $\R^2$, by the Lebesgue decomposition theorem we have $Du=\nabla u\mathscr{L}^2+D^s u$, where $\nabla u\in L^1(\Omega;\R^{2\times2})$ and  $D^s u\perp\mathscr{L}^2$.
The symbol $|Du|(\Omega)$ stands for the total variation of $Du$ 
(see \cite[Definition 3.4, pag. 119]{AFP})  
with $\vert \cdot \vert$ the Frobenius norm.
 
\begin{Definition}[\textbf{Strict convergence}]
	Let $\mappa\in BV(\Omega;\R^2)$ and $(\mappa_k)
\subset BV(\Omega;\R^2)$. 
We say that $(\mappa_k)$ converges
 to $\mappa$ strictly $BV$, 
if
	$$
	\mappa_k\xrightarrow{L^1}\mappa\qquad\mbox{and}\qquad|D\mappa_k|
(\Omega)\rightarrow|D\mappa|(\Omega).
	$$
\end{Definition}
Now, let $B_\radius$ be the disk of $\R^2$ centered at the origin of radius $\radius>0$. If $u\in BV(B_\radius;\R^2)$, by Lebesgue differentiation theorem and Fubini theorem, for almost every $r<\radius$ the restriction $u\mres\partial B_r$ is well defined and independent of the representative of $u$, since it coincides with the trace of $u$ on $\mathcal{H}^1$-almost every point of $\partial B_r$.
In particular,  for almost every $r<\radius$, we can define the total variation of $u\mres\partial B_r$ as 
\begin{equation}\label{var on circ}
|D(u\,\mres\partial B_r)|(\partial B_r):=\sup\left\{\int_0^{2\pi}\!\!\bar{u}(r,\theta)\cdot f'(\theta)d\theta; f\in C^1([0,2\pi];\overline B_1(0)),f(0)=f(2\pi),f'(0)=f'(2\pi)\!\right\}
\end{equation}
which turns out to be finite (see Lemma \ref{lem:inheritance circ}), giving that $u\mres\partial B_r\in BV(\partial B_r;\R^2)$, for almost every $r<\radius$. Here $\bar u(r,\theta):=u(r\cos\theta,r\sin\theta)$ for every $r\in(0,\radius],\theta\in[0,2\pi]$.\\

We want to relate this quantity with the notion of \textit{tangential total variation}.
\begin{Definition}[{\bf Tangential total variation in an annulus}]\label{def:annulus}
For $x=(x_1,x_2)\in\R^2\setminus\{(0,0)\}$, set 
$\tau(x)=\frac{1}{|x|}(-x_2,x_1)$. Let $0< \eps< \radius$ and 
$A_{\eps,\radius}:=B_\radius\setminus \overline{B_\eps}$ be an 
annulus around $0$. We define the tangential total variation of $u\in BV(A_{\eps,\radius};\R^2)$ 
as the total variation of the Radon measure $D_\tau u:=Du\tau$, namely
\begin{equation}\label{tangential var ann}
|D_\tau u|(A_{\eps,\radius})=|Du\tau|(A_{\eps,\radius})=\sup\Big\{-\int_{A_{\eps,\radius}}u\cdot(\nabla g\tau) ~dx: g\in C^1_c(A_{\eps,\radius};\overline B_1(0))
\Big\}.
\end{equation}
\end{Definition}
\noindent
The last equality in \eqref{tangential var ann} is justified since $\tau\in C^\infty(A_{\eps,\radius};\R^2)$ satisfies $\text{\rm div}\tau=0$ everywhere, so for any $g=(g^1,g^2)\in C^1_c(A_{\eps,\radius};\R^2)$ we have
\begin{equation}\nonumber
\begin{aligned}
 -\int_{A_{\eps,\radius}}u\cdot(\nabla g\tau) ~dx
&=-\int_{A_{\eps,\radius}}u^1\nabla g^1\cdot\tau ~dx-\int_{A_{\eps,\radius}}u^2\nabla g^2\cdot\tau ~dx
\\
&=-\int_{A_{\eps,\radius}}u^1\mathrm{div}( g^1\tau) ~dx-\int_{A_{\eps,\radius}}u^2\mathrm{div}( g^2\tau )~dx
\\
&=\int_{A_{\eps,\radius}}g^1\tau\cdot dDu^1+\int_{A_{\eps,\radius}}g^2\tau\cdot dDu^2
=\int_{A_{\eps,\radius}}g\cdot (dDu)\tau=\langle Du\tau,g\rangle.
\end{aligned}
\end{equation}
This computation shows that $|D_\tau u|(A_{\eps,\radius})\leq|Du|(A_{\eps,\radius})$, since $|\tau|\leq1$, and also that \eqref{tangential var ann} is compatible with the case $u\in W^{1,1}(A_{\eps,\radius};\R^2)$, 
where simply $|D_\tau u|(A_{\eps,\radius})=\int_{A_{\eps,\radius}}|\nabla u\tau|~dx$. 
Moreover, $Du=\frac{Du}{|Du|}|Du|$ by polar decomposition, 
so that
$$
\langle Du\tau,g\rangle=\int_{A_{\eps,\radius}}g\cdot (dDu)\tau=\int_{A_{\eps,\radius}}g\cdot \left(\frac{Du}{|Du|}d|Du|\right)\tau=\int_{A_{\eps,\radius}}g\cdot \left(\frac{Du}{|Du|}\tau\right)d|Du|\quad\forall g\in C_c^1(A_{\eps,\radius};\R^2),
$$
giving that 
\begin{align}\label{polar dec}
D_\tau u=Du\tau=\frac{Du}{|Du|}\tau|Du|.
\end{align}
In \cite{BCS2}, the following slicing result for the strict convergence is proven.
\begin{Lemma}[\textbf{Inheritance of strict convergence to circumferences}]\label{lem:inheritance circ}
Let $u\in BV(B_\radius;\R^2)$. 
Suppose that $(v_k)\subset C^1(B_\radius;\R^2)$ is a sequence 
converging to $u$ strictly $BV(B_\radius;\R^2)$. 
Then, for almost every $r\in(0,l)$, there exists a subsequence $(v_{k_h})
\subset (v_k)$, depending on $r$, such that
\begin{equation}\label{thesis}
v_{k_h}\mres \partial B_r\rightarrow u\mres \partial B_r\quad\mbox{strictly }BV(\partial B_r;\R^2).
\end{equation}
\end{Lemma}
In the proof of Lemma \ref{lem:inheritance circ} a useful Coarea-type formula is provided:
\begin{Lemma}\label{Fubini formula}
Let $u\in BV(B_\radius;\R^2)$. Then
 \begin{align}\label{Fub form}
|D_\tau u|(A_{\eps,\radius})=\int_\eps^\radius|D(u\mres\partial B_r)|(\partial B_r)~dr.
\end{align}
\end{Lemma}
This formula allows us to define a notion of tangential total variation for $u\in BV(B_\radius;\R^2)$ on the whole $B_\radius$, since the right hand side of \eqref{Fub form} is monotone non-increasing and equibounded w.r.t. $\eps$.
\begin{Definition}[{\bf Tangential total variation in $B_\radius$}]\label{def:tang tot var}
Let $\tau$ and $A_{\eps,\radius}$ as in Definition \ref{def:annulus}. We define the tangential total variation of $u\in BV(B_\radius;\R^2)$ as 
\begin{equation}\label{tangential var}
|D_\tau u|(B_\radius):=\lim_{\eps\to0^+}|D_\tau u|(A_{\eps,\radius})=\int_0^\radius|D(u\mres\partial B_r)|(\partial B_r)~dr.
\end{equation}
\end{Definition}
\noindent

\begin{Proposition}\label{homog extension is BV}
Let $\gamma\in BV(\Suno;\R^2)$ and $u$ be defined as in \eqref{homog map def}. Then $u\in BV(B_\radius;\R^2)$ and
\begin{align}\label{var u}
|Du|(B_\radius)=\radius|\dot{\gamma}|(\Suno).
\end{align}
 Moreover, 
\begin{align}\label{var split}
\int_{B_\radius}|\nabla u|dx=\radius\int_{\Suno}|\dot\gamma^a|dy,\qquad|D^su|(B_\radius)=\radius|\dot\gamma^s|(\Suno).
\end{align}
\end{Proposition}
\begin{proof}
Since $\bar{u}$ does not depend on $\rho$, by \eqref{var on circ}, we have $|D(u\mres\partial B_r)|(\partial B_r)=|\dot\gamma|(\Suno)$. So, thanks to \eqref{Fub form}, in order to prove \eqref{var u} it is enough to show that the variation of $u$ is purely tangential, namely $|Du|(B_\radius)=|D_\tau u|(B_\radius)$. To this purpose, set $\nu(x)=\frac{x}{|x|}$, $x\neq0$, and define the measure $D_\nu u:=Du\nu$ on the annulus $A_{\eps,\radius}$, i.e.
$$
\langle D_\nu u,g\rangle=\int_{A_{\eps,\radius}}u^1\mathrm{div}(g^1\nu)dx+\int_{A_{\eps,\radius}}u^2\mathrm{div}(g^2\nu)dx\quad\forall g\in C^1_c(A_{\eps,\radius};\R^2).
$$
By polar decomposition of vector valued Radon measure, for $i=1,2$ we have
$$Du^i=\frac{Du^i}{|Du^i|}|Du^i|=\left(\frac{Du^i}{|Du^i|}\cdot\tau\tau+\frac{Du^i}{|Du^i|}\cdot\nu\nu\right)|Du^i|=Du^i\cdot\tau\tau+Du^i\cdot\nu\nu.$$
Let us prove that $Du^i\cdot\nu=0$ on $A_{\eps,\radius}$, for $i=1,2$. Recall that $\bar\gamma(\theta):=\bar{u}(\rho,\theta)=\gamma(\cos\theta,\sin\theta)$. Let $\psi\in C^1_c(A_{\eps,\radius})$, then since div$\nu=\frac{1}{|x|}$ in $A_{\eps,\radius}$, we get
\begin{align*}
\langle Du^i\cdot\nu,\psi\rangle&=\int_{A_{\eps,\radius}}u^i\mathrm{div}(\psi\nu)dx=\int_{A_{\eps,\radius}}u^i\psi\mathrm{div}\nu dx+\int_{A_{\eps,\radius}}u^i\nabla\psi\cdot\nu dx\\
&=\int_\eps^\radius\int_0^{2\pi}\rho \bar{u}^i(\rho,\theta)\bar{\psi}(\rho,\theta)\frac{1}{\rho}d\rho d\theta+\int_\eps^\radius\int_0^{2\pi}\rho \bar{u}^i(\rho,\theta)\partial_\nu\bar{\psi}(\rho,\theta)d\rho d\theta\\
&=\int_\eps^\radius\int_0^{2\pi}\bar{\gamma}^i(\theta)\bar{\psi}(\theta)d\rho d\theta+\int_\eps^\radius \bar{\gamma}^i(\theta)\left[\int_0^{2\pi}\rho\partial_\nu\bar{\psi}(\rho,\theta)d\rho\right] d\theta\\
&=\int_\eps^\radius\int_0^{2\pi}\bar{\gamma}^i(\theta)\bar{\psi}(\rho,\theta)d\rho d\theta-\int_\eps^\radius\int_0^{2\pi}\bar{\gamma}^i(\theta)\bar{\psi}(\rho,\theta)d\rho d\theta=0.
\end{align*}
We infer that $Du=(Du\tau)\otimes\tau$ on $A_{\eps,\radius}$. Now, since $|(Du\tau)\otimes\tau|(A_{\eps,\radius})\leq|D_\tau u|(A_{\eps,\radius})\leq|Du|(A_{\eps,\radius})$, passing to the limit as $\eps\to0^+$, we conclude that $|Du|(B_\radius)=|D_\tau u|(B_\radius)=\radius|\dot\gamma|(\Suno)$.\\
Finally, in polar coordinates
\begin{align}\label{nabla u}
\nabla u(\rho\cos\theta,\rho\sin\theta)=\frac{\dot{\bar\gamma}^a(\theta)}{\rho}\qquad\mbox{a.e. } \rho\in(0,\radius],\theta\in[0,2\pi],
\end{align}
so that
$$
\int_{B_\radius}|\nabla u|dx=\int_0^\radius\int_0^{2\pi}\rho\frac{|\dot{\bar\gamma}^a(\theta)|}{\rho}d\theta d\rho=\radius\int_{\Suno}|\dot\gamma^a|dy
$$
and
$$
|D^su|(B_\radius)=|Du|(B_\radius)-\int_{B_\radius}|\nabla u|dx=\radius|\dot\gamma|(\Suno)-\radius\int_{\Suno}|\dot{\gamma}^a|dy=\radius|\dot\gamma^s|(\Suno).
$$
\end{proof}

\subsection{Further properties in dimension 1}
In \cite[Proposition 2.4]{BCS} the following is proved:

\begin{Lemma}\label{lem:strict_implies_uniform_1}
Let $(\gamma_k) \subset W^{1,1}([a,b];\R^2)$ be a sequence
converging strictly $BV([a,b];\R^2)$ to $\gamma \in 
W^{1,1}([a,b];\R^2)$. Then $\gamma_k\rightarrow \gamma$ uniformly in $[a,b]$. 
\end{Lemma}

The same result holds also in case $\gamma\in BV([a,b];\R^2)$, but only on those compact subsets of $[a,b]$ which do not intersect the jump set $J_\gamma$. 

\begin{Lemma}\label{lem:strict_implies_uniform_2}
Let $(\gamma_k) \subset W^{1,1}([a,b];\R^2)$ be a sequence
converging strictly $BV([a,b];\R^2)$ to $\gamma \in 
BV([a,b];\R^2)$. Then, for every compact subset $K\subset[a,b]\setminus J_\gamma$, we have that 
\begin{align}\label{claim}
\gamma_k\rightarrow\gamma\quad\mbox{ uniformly in }K\quad\mbox{ as }k\rightarrow+\infty.
\end{align}
\end{Lemma}

\begin{proof}
By contradiction, up to a not relabeled subsequence, we may suppose
$$
\exists\delta>0 \quad\exists(\tau_k)\subset K\quad\exists k_0\in\N:\quad|\gamma_k(\tau_k)-\gamma(\tau_k)|>\delta \quad\forall k\geq k_0,
$$
and there exists $\bar \tau\in K$ such that $\tau_k\rightarrow\bar \tau$ as $k\rightarrow+\infty$, since $K$ is compact. Now, consider an open interval $E\subset[a,b]$ such that\footnote{If $\bar\tau=a$ or $\bar\tau=b$, $E$ is a semi-open interval.} $\bar \tau\in E$, $\partial E\subset[a,b]\setminus J_\gamma$, and $|\dot\gamma|(E)<\frac{\delta}{4}$. Such an interval $E$ exists because $|\dot\gamma|(\{\bar\tau\})=0$. By hypothesis on strict convergence, since $|\dot\gamma|(\partial E)=0$, we have
$$
\lim_{k\rightarrow+\infty}\int_{E}|\dot\gamma_k|dt=|\dot\gamma|(E).
$$
So, we can find an index $k_1\in\N$ such that $k_1\geq k_0$ and $\int_E|\dot\gamma_k|dt<\frac{\delta}{2}$, for every $k\geq k_1$. Moreover, there exists $k_2\in\N$, $k_2\geq k_1$, such that $\tau_k\in E$ for every $k\geq k_2$. Now fix $F\subset E$ such that $|F|=|E|$ and $\gamma_{|F}$ can be identified with its natural continuous representative. Pick a point $z\in F$, then
\begin{align*}
|\gamma_k(z)-\gamma(z)|&\geq-|\gamma_k(z)-\gamma_k(\tau_k)|+|\gamma_k(\tau_k)-\gamma(\tau_k)|-|\gamma(\tau_k)-\gamma(z)|\\
&\geq-\left|\int_{\tau_k}^{z}|\dot\gamma_k|dt\right|+\delta-|\dot\gamma|(E)\geq-\int_E|\dot\gamma_k|dt+\delta-\frac{\delta}{4}\\
&\geq-\frac{\delta}{2}+\frac{3}{4}\delta=\frac{\delta}{4}.
\end{align*} 
Therefore, $(\gamma_k)$ does not converge to $\gamma$ pointwise at any point of $F$, which leads to a contradiction with the fact that $\gamma_k\rightarrow\gamma$ in $L^1([a,b])$. So, \eqref{claim} is proved.
\end{proof}

An immediate consequence of Lemma \ref{lem:strict_implies_uniform_2} is that the uniform convergence takes place on the full interval if $J_\gamma=\varnothing$. Precisely the following holds.

\begin{Corollary}\label{cor:strict_implies_uniform_2}
Let $(\gamma_k) \subset W^{1,1}([a,b];\R^2)$ be a sequence
converging strictly $BV([a,b];\R^2)$ to $\gamma \in 
C([a,b];\R^2)\cap BV([a,b];\R^2)$. Then, 
\begin{align*}
\gamma_k\rightarrow\gamma\quad\mbox{ uniformly }\mbox{ as }k\rightarrow+\infty.
\end{align*}
\end{Corollary}

 This is clearly impossible if $J_\gamma$ is non-empty, but becomes true (up to extracting a subsequence) if we suitably reparametrize $\gamma_k$ and if instead of $\gamma$ we consider its "completed curve" $\widetilde\gamma$, obtained by filling the jumps with line segments. This is the content
of \cite[Lemma 2.7, Corollary 2.8]{BCS2}, where the authors prove the result for $\gamma\in SBV([a,b];\R^2)$, which is allowed to jump on a finite number of points. Our goal is to provide a further improvement of this result, namely, when $\gamma$ is just a function of bounded variation.\\
To this purpose, suppose that $\gamma\in BV([a,b];\R^2)$. Then, it is well known that $J_\gamma$ is at most countable. So, let $\{t_i\}_{i\in\N}$ be an enumeration\footnote{If the number of jumps is finite, then $\{t_i\}$ is definitively constant.} of $J_\gamma$ and $\gamma^{\pm}(t_i)$ be the traces of $\gamma$ at $t_i$. We want to associate to $\gamma$ a unique continuous curve $\widetilde\gamma$  which "completes" the image of $\gamma$ by means of segments connecting $\gamma^-(t_i)$ to $\gamma^+(t_i)$. In particular, we require that $\widetilde\gamma$ has the same total variation $L$ of $\gamma$ and is compatible with the approximation via strict $BV$-convergence. Precisely we show the following result.

\begin{Lemma}\label{strict implies uniform 3}
Suppose that $(\gamma_k)
	\subset W^{1,1}([a,b] ;\R^2)$ is a sequence 
	converging strictly $BV([a,b] ;\R^2)$ to $\gamma\in BV([a,b] ;\R^2)$.
	Then there exist:
	\begin{itemize}
		\item[(a)] a curve $\widetilde\gamma\in \mathrm{Lip} ([a,b];\R^2)$, 
		\item[(b)]
		a subsequence $(k_j)$ and Lipschitz strictly increasing 
surjective functions $\repa_{k_j}:[a,b]\rightarrow[a,b]$ for any $j \in \mathbb N$, 
with $\sup_j \Vert \dot \repa_{k_j}\Vert_\infty < +\infty$,
	\end{itemize}
	such that
	\begin{align}\label{uniform_rescalings}
	\lim_{\indice \to +\infty} 
\gamma_{k_j}\circ \repa_{k_j} = \widetilde\gamma\quad\mbox{ uniformly in }[a,b].
	\end{align}
Moreover, $\widetilde\gamma$ does not depend on the approximating sequence $\gamma_k$, in the sense that if $(\eta_k)\subset W^{1,1}([a,b];\R^2)$ is another sequence converging strictly $BV([a,b];\R^2)$ to $\gamma$, then the corresponding $\widetilde\eta\in\mathrm{Lip}([a,b];\R^2)$ coincides with $\widetilde\gamma$.
\end{Lemma}

\begin{proof}
	The lengths $L_k$ of $\gamma_k$ and $L$ of $\gamma$ are given by 
	\begin{align*}
		&L_k=\int_{a}^b|\dot{\gamma}_k|~d\tau,\qquad L=|\dot{\gamma}|([a,b]).
	\end{align*}
Since, by assumption, $\gamma_k\rightarrow\gamma$ strictly $BV([a,b];\R^2)$, we have that $L_k\rightarrow L$ as $k\rightarrow+\infty$. For every $k\in\N$, define
	\begin{equation}\label{eq:s_k}
s_k:[a,b]\rightarrow[0,L], \qquad 
		\begin{aligned}
			s_k(t):=
			\frac{L}{L_k+b-a}\int_{a}^t\Big(|\dot\gamma_k(\tau)|+1\Big)~d\tau,
		\end{aligned}
	\end{equation}
	with Lipschitz inverse $\alpha_k:=s_k^{-1}:[0,L]\rightarrow[a,b]$. Notice that 
\begin{align}\label{hk derivative}
\dot{\alpha}_k(s)=\frac{1}{\dot s_k(\alpha_k(s))}=\frac{L_k+b-a}{L}\cdot\frac{1}{|\dot\gamma_k(\alpha_k(s))|+1}\leq\frac{L_k+b-a}{L}\leq C\qquad {\rm for~a.e.}~
	s\in [0, L],
\end{align}
for some constant $C>0$ independent of $k$.
Define 
	\begin{align*}
\bar\gamma_k:[0,L]\rightarrow\R^2, \qquad 
		\bar\gamma_k(s):=\gamma_k(\alpha_k(s))
		\qquad\forall s\in [0,L].
	\end{align*} 
Since
	$$\left|\frac{d\bar\gamma_k}{ds}(s)\right|\leq 
	\frac{|\dot\gamma_k(\alpha_k(s))|}{|\dot s_k(\alpha_k(s))|}\leq\frac{L_k+b-a}{L}\leq C \qquad {\rm for~a.e.}~
	s\in [0, L],
	$$	
	the sequence $(\bar\gamma_k)$ is 
	bounded in $W^{1,\infty}([0,L];\R^2)$. Thus, there exists a subsequence $(k_j)\subset(k)$ and $\bar\gamma\in W^{1,\infty}([0,L];\R^2)$ such that 
	\begin{align}\label{thesis1}
		\bar \gamma_{k_j}\rightharpoonup\bar\gamma\text{ weakly* in }W^{1,\infty}([0,L];\R^2)\text{ and uniformly in }[0,L].
	\end{align}
Then, we conclude by defining $\widetilde\gamma$ and $h_k$ as the composition of $\bar\gamma$ and $\alpha_k$ with an affine increasing diffeomorphism $\psi:[a,b]\to[0,L]$.\\
It remains to show the indipendence of $\bar\gamma$ from the sequence $\gamma_k$. So, suppose that $\eta_k\in W^{1,1}([a,b];\R^2)$ converges to $\gamma$ strictly $BV([a,b];\R^2)$. Let $\sigma_k:[a,b]\rightarrow[0,L]$ be defined as $s_k$ with $\eta_k$ in place of $\gamma_k$ and $\beta_k:=\sigma_k^{-1}:[0,L]\rightarrow[a,b]$ its (equi-)Lipschitz inverse. As before, we obtain that there exists $(k_h)\subset(k)$ and $\bar\eta$ such that 
$$
\bar \eta_{k_h}\rightharpoonup\bar\eta\text{ weakly* in }W^{1,\infty}([0,L];\R^2)\text{ and uniformly in }[0,L].
$$
Observe that for any open interval $J\subseteq [0,L]$,
	$$\int_{J}|\dot{\bar\gamma}|ds\leq
	\liminf_{k\rightarrow+\infty}\int_{J}|\dot{\bar\gamma}_{k}|ds \leq |J| \liminf_{k\rightarrow +\infty}\frac{L_k+b-a}{L}=\frac{L+b-a}{L}|J|,
	$$
and thus
	\begin{equation}\label{eq:subunitary}
		|\dot{\bar\gamma}|\leq 1+\frac{b-a}{L}\text{ a.e. in }[0,L].
	\end{equation}

Now, fix $i\in\N$ and take any sequence ${(t_{i,j}^\pm)}_j\subset[a,b]\setminus J_\gamma$ such that $t_{i,j}^-\nearrow t_i^-$ and $t_{i,j}^+\searrow t_i^+$ as $j\rightarrow+\infty$. 
	By Lemma \ref{lem:strict_implies_uniform_2} and definition of $\gamma^{\pm}$, we have 
	\begin{align}\label{limit tij}
\lim_{j \to +\infty}\gamma_{k_j}(t_{i,j}^\pm)= \gamma^\pm(t_i).
\end{align}
Setting
\begin{equation}\label{skj}
	\begin{aligned}
		r_{i,j}^-
		:=&
		s_{k_j}(t_{i,j}^-)=\frac{L}{L_{k_j}+b-a}\int_{a}^{t_{i,j}^-}\big(|\dot\gamma_{k_j}|+1 \big)~d\tau,
		\\
		r^+_{i,j}
		:=& 
		s_{k_j}(t_{i,j}^+)=
\frac{L}{L_{k_j}+b-a}\int_{a}^{t_{i,j}^+}\big(|\dot\gamma_{k_j}|+1 \big)~d\tau,
	\end{aligned}
\end{equation}
	we have 
	\begin{equation}\label{eq:conv_rij}
		\begin{aligned}
			& 
			\lim_{j \to +\infty}
			r_{i,j}^- = \frac{L}{L+b-a}|\dot\gamma|([a,t_i))=:s^-(t_i),
			\\
			&\lim_{j \to +\infty}
			r_{i,j}^+= \frac{L}{L+b-a}|\dot\gamma|([a,t_i])= \frac{L}{L+b-a}\big[|\dot\gamma|([a,t_i))+|\gamma^+(t_i)-\gamma^-(t_i)|\big]=:s^+(t_i).
		\end{aligned}
	\end{equation}
	As a consequence of \eqref{thesis1}, \eqref{limit tij}, and \eqref{eq:conv_rij},
	we get
	$$	
	\bar\gamma(s^\pm(t_i))\leftarrow\bar\gamma_{k_\indice}(
	r_{i,j}^\pm)=\gamma_{k_j}(\alpha_{k_j}(
	r_{i,j}^\pm
	))=\gamma_{k_j}(t_{i,j}^\pm)
	\rightarrow 
	\gamma^\pm(t_i)\quad {\rm ~as~} \indice \to +\infty.
	$$
	Therefore the curve $\bar\gamma$ 
	maps the segment $[s^-(t_i),s^+(t_i)]$ into a curve joining
$\gamma^-(t_i)$ 
	and $\gamma^+(t_i)$. Now, since $s^+(t_i)-s^-(t_i)=\frac{L}{L+b-a}|\gamma^+(t_i)-\gamma^-(t_i)|$, from \eqref{eq:subunitary}
	we conclude that $\bar \gamma$ 
	coincides with the $\left(1+\frac{b-a}{L}\right)$-speed parametrization $\ell_i$ of the segment 
joining $\gamma^-(t_i)$ and $\gamma^+(t_i)$
	on $[s^-(t_i),s^+(t_i)]$. Hence we have shown that for every $i\in\N$
	\begin{align*}
		\gamma_{k_j}\circ \alpha_{k_j}\rightarrow  \ell_i\text{ uniformly in }[s^-(t_i),s^+(t_i)] {\rm ~as~} \indice \to +\infty.
	\end{align*}
An analogous conclusion holds also for $\eta_{k_h}$: indeed, let $\sigma_{k_h}(t_{i,h}^\pm)$ be as in \eqref{skj} but with $\eta_{k_h}$ in place of $\gamma_{k_j}$, then it is clear that $\sigma_{k_h}(t_{i,h}^\pm)\to s^\pm(t_i)$ as $h\to+\infty$ and so
\begin{align*}
		\eta_{k_h}\circ \beta_{k_h}\rightarrow  \ell_i\text{ uniformly in }[s^-(t_i),s^+(t_i)] {\rm ~as~} h \to +\infty.
	\end{align*}
	Therefore, $\bar\eta=\bar\gamma$ on $S=\cup_{i\in\N}S_i$, where $S_i:=[s^-(t_i),s^+(t_i)]$. It remains to show that $\bar\eta=\bar\gamma$ on $[0,L]\setminus S$.\\
By \eqref{hk derivative}, up to extract a not relabeled subsequence, we can assume that there exists $
\alpha\in W^{1,\infty}([0,L])$ such that
\begin{align}\label{h unif}
\alpha_{k_j}\rightarrow \alpha \quad\mbox{uniformly in } [0,L] \mbox{ as }j\to+\infty
\end{align}
and, for the same reason, there exists $\beta\in W^{1,\infty}([0,L])$ such that
\begin{align}\label{g unif}
\beta_{k_h}\rightarrow \beta \quad\mbox{uniformly in }  [0,L] \mbox{ as }h\to+\infty.
\end{align}
From Lemma \ref{lem:strict_implies_uniform_2}, we deduce that $\bar\gamma=\gamma\circ \alpha$ on every compact subset $H\subset[0,L]\setminus S$. But, since $\alpha$ does not depend on the compact $H$, we deduce that $\bar\gamma=\gamma\circ \alpha$ on $[0,L]\setminus S$. In the same way, we infer that $\bar\eta=\gamma\circ \beta$ on $[0,L]\setminus S$.
	Let us show that $\alpha=\beta$ on $[0,L]\setminus S$. Indeed, notice that by definition of $s_k$,
$$
s_k(t)\rightarrow s(t):=\frac{L}{L+b-a}(t-a+|\dot\gamma|([a,t]))\quad\forall t\in[a,b]\setminus J_\gamma.
$$
The map $s:[a,b]\rightarrow[0,L]$ is strictly increasing with jumps at each point of $J_\gamma$. Notice that the traces of $s$ at every $t_i\in J_\gamma$ are exactly the numbers $s^\pm(t_i)$ in \eqref{eq:conv_rij}. We claim that $\alpha=s^{-1}$ on $[0,L]\setminus S$. Indeed, by \eqref{h unif} we have that for every $t\in[a,b]\setminus J_\gamma$
$$
t=\alpha_{k_j}(s_{k_j}(t))\rightarrow \alpha(s(t))\quad\mbox{as }j\to+\infty,
$$
then $\alpha=s^{-1}$ on $s([a,b]\setminus J_\gamma)=[0,L]\setminus S$. In the same way, using $\eqref{g unif}$ one can prove that $\beta=s^{-1}$ on $[0,L]\setminus S$ and we conclude the proof.

\begin{Remark}
From the previous proof, we deduce that the "completed" curve $\widetilde\gamma$ does not depend on the subsequence of the approximating sequence $\gamma_k$. Moreover, we do not need to discuss the dependence on the reparametrization $h_k$, because, for our purpose, we shall consider in the sequel a Plateau-type problem associated to $\gamma_k$ which is independent of the reparametrization of the curve.
\end{Remark}

\end{proof}

\subsection{Planar Plateau-type problem}
\label{subsec:planar_Plateau_type_problem}
In \cite{BCS2}, the authors consider the following planar Plateau-type problem spanning a closed Lipschitz curve $\varphi:\Suno\to\R^2$ (see also \cite{Pa} and \cite{FFM}):
\begin{align}\label{Plateau1}
P(\varphi)=\inf\left\{\int_{B_1}|Jv|~dx: \,v\in \mathrm{Lip}(B_1;\R^2), v_{|\partial B_1}=
\homogeneousmap
\right\}
\end{align}
 and the corresponding relaxation problem for a general $BV$-curve $\gamma:\Suno\to\R^2$:
\begin{align}\label{Plateau_rel1}
	\rilP(\gamma)
:=\inf\left\{\liminf_{n\rightarrow +\infty}P(\varphi_n): 
\varphi_n\in \mathrm{Lip}(\Suno;\R^2), \varphi_n\rightarrow\gamma\text{ strictly }BV(\Suno;\R^2)
\right\}.
\end{align}
They show that, for $\varphi\in$ Lip$(\Suno;\R^2)$, $P(\varphi)$ is invariant under rescaling of the integration domain, precisely if $r>0$ and 
\begin{align}\label{def resc}
\varphi_r(y):=\varphi\left(\frac{y}{r}\right)\quad y\in\partial B_r, 
\end{align}
then
\begin{align}\label{resc}
P(\varphi)=P(\varphi_r)=\inf\left\{\int_{B_r}|Jv|~dx: \,v\in \mathrm{Lip}(B_r;\R^2), v_{|\partial B_r}=
\homogeneousmap_r
\right\}.
\end{align}
$P(\cdot)$ is also invariant under reparametrization of the boundary datum, namely
\begin{align}\label{inv}
P(\varphi)=P(\varphi\circ h)\qquad \forall h:\Suno\to\Suno \mbox{ Lipschitz homeomorphism}. 
\end{align}
 Moreover, the following continuity result for $P(\cdot)$ holds.
\begin{Lemma}[\textbf{Continuity of $P$}]\label{lem:continuity_of_P}
Let $\homogeneousmap\in\mathrm{Lip}(\Suno;\R^2)$ and suppose that  
$(\varphi_k)_k\subset\mathrm{Lip}(\Suno;\R^2)$ is such that $$\varphi_k\rightarrow\homogeneousmap\quad\mbox{uniformly}\quad\mbox{and }\quad \sup_k|\dot\varphi_k|(\Suno)<+\infty.$$ Then $P(\varphi_k)\rightarrow P(\homogeneousmap)$.
\end{Lemma}
In \cite[Lemma 2.14]{BCS2}, the authors show that if $\gamma\in SBV(\Suno;\R^2)$ has a finite number of jump points, then 
$\rilP(\gamma)=P(\widetilde\gamma)$, 
where $\widetilde\gamma$ is the Lipschitz curve\footnote{$\Suno$ is identified with $[0,2\pi]$.} of Lemma \ref{strict implies uniform 3} associated to $\gamma$. We want to extend this result to the case $\gamma\in BV(\Suno;\R^2)$.

\begin{Lemma}\label{lem_plateaurel}
	Let $\gamma\in BV(\Suno;\R^2)$ and $\widetilde\gamma:\Suno\rightarrow\R^2$ be the corresponding Lipschitz curve 
of Lemma \ref{strict implies uniform 3}. Then 
	\begin{align}
	\rilP(\gamma)=P(\widetilde\gamma).
	\end{align}
\end{Lemma}
\begin{proof}
	Let $(\gamma_k)_k\subset\mathrm{Lip}(\Suno;\R^2)$ be a 
sequence converging strictly to $\gamma$. Let us consider a not-relabeled subsequence of $(\gamma_k)_k$; by Lemma \ref{strict implies uniform 3} there are a further subsequence $(\gamma_{k_j})_j$ and 
Lipschitz reparametrizations $\widetilde\gamma_{k_j}=\gamma_{k_j}\circ h_{k_j}\in\mathrm{Lip}(\Suno;\R^2)$ 
of $\gamma_{k_j}$ such that $\widetilde\gamma_{k_j}\rightarrow
\widetilde\gamma$ 
uniformly as $j\rightarrow +\infty$.  Moreover, since by Lemma \ref{strict implies uniform 3}(b) the homeomorphism $\repa_{k_j}$ can be chosen with uniformly bounded 
Lipschitz constant, it follows that $\widetilde\gamma_{k_j}$ has uniformly 
bounded total variation. Hence it follows from Lemma 
\ref{lem:continuity_of_P} that 
$P(\widetilde\gamma_{k_j})\rightarrow P(
\widetilde\gamma)$ as $j\rightarrow +\infty$. Thanks to \eqref{inv}, we have also
	$P(\gamma_{k_j})\rightarrow P(\widetilde\gamma)$ as $j\rightarrow +\infty$. Then, since this argument holds for any subsequence of $(\gamma_k)$, we conclude that the whole sequence satisfies
	$P(\gamma_{k})\rightarrow P(\widetilde\gamma)$. Finally, since by Lemma \ref{strict implies uniform 3} $\widetilde\gamma$ does not depend on the approximating sequence, we can repeat the previous argument for another sequence $(\eta_k)\subset$ Lip$(\Suno;\R^2)$ converging strictly to $\gamma$, obtaining that $P(\eta_k)\to P(\widetilde\gamma)$. Therefore, we conclude $\rilP(\gamma)=P(\widetilde\gamma)$.
\end{proof}

As a consequence of the argument in the proof of Lemma \ref{lem_plateaurel}, we easily infer the following continuity property:

\begin{Corollary}\label{cor_continuity_P}
	Let $\gamma\in BV(\Suno;\R^2)$ and $\widetilde\gamma$ be as in Lemma \ref{strict implies uniform 3}, and assume that $(\gamma_k)_k\subset \mathrm{Lip}(\Suno;\R^2)$ is a sequence converging strictly to $\gamma$. Then
	$$\lim_{k\rightarrow +\infty}P(\gamma_k)=\rilP(\gamma)=P(\widetilde\gamma).$$ 
\end{Corollary}

\medskip

\vspace{5mm}

\section{Relaxation results}\label{sec:relaxation}
In this section, we extend the results in \cite[Sec.4]{BCS2} to homogeneous maps as in Definition \ref{hom map}. \\
To start with, it is worth to consider the case of homogeneous extension $u$ of a Lipschitz map $\varphi:\Suno\to\R^2$, namely
\begin{align}\label{hom map Lip}
u(x)=\varphi\left(\frac{x}{|x|}\right)\quad\forall x\in B_\radius\setminus\{(0,0)\}.
\end{align}

 In this case, clearly $u\in W^{1,1}(B_\radius;\R^2)$ and $\int_{B_\radius}|\nabla u|dx=\radius\int_{\Suno}|\dot\varphi|d\mathcal{H}^1$.
The following result extends the validity of \cite[Thm.1]{Pa} also for the relaxation with respect to the strict $BV$-convergence.
\begin{Theorem}\label{TV and Plateau}
Suppose that $\varphi:\Suno\rightarrow\R^2$ is Lipschitz continuous and let $u$ be defined as in \eqref{hom map Lip}. Then 
\begin{equation}\label{eq:ril_TVJ_hom}
\overline{\totvarjac}_{BV}(u,B_\radius)
=
P(\homogeneousmap).
\end{equation}
\end{Theorem}

\begin{proof}
Let us show the upper bound inequality.
Following the proof of Theorem 1 in \cite{Pa}, for $k\geq 2$, a recovery sequence $v_k\in$ Lip$(B_\radius;\R^2)$ is given by
\begin{equation}\label{rec}
v_k(x)=\left\{
\begin{aligned}
&u(x)\quad \mbox{if } |x|> \radius/k,
\\
&(v)_{\frac\radius k}(x)\quad \mbox{if } |x|\leq \radius/k,
\end{aligned}
\right.
\end{equation}
where $v\in \mathrm{Lip}(B_1;\R^2)$ is any map with $v=\homogeneousmap$ on $\partial B_1$ and $(v)_{\frac\radius k}(x):=v\left(\frac{k}{\radius}x\right)$ for $x\in B_\frac{\radius}{k}$. It is not difficult to see that $v_k\rightarrow u$ strongly in $W^{1,1}(B_\radius;\R^2)$ (and hence strictly $BV(B_\radius;\R^2)$). Moreover, by change of variable 
\begin{align}\label{lim jac}
\int_{B_\radius}|{Jv_k}|dx=\int_{B_\frac{\radius}{k}}|{J(v)_{\frac\radius k}}|dx=\int_{B_1}|Jv|dx \quad \forall k\in \N.
\end{align}
Finally, we get
$$
\overline{{\totvarjac}}_{BV}(u,B_\radius)\leq\liminf_{k\rightarrow+\infty}\int_{B_\radius}|{Jv_k}|dx=\int_{B_1}|Jv|dx
$$
for any $v\in \mathrm{Lip}(B_1;\R^2)$ such that $v=
\homogeneousmap
$ on $\partial B_1$, so we deduce that $\overline{{\totvarjac}}_{BV}(u,B_\radius)\leq P(\varphi)$.\\
Now let us prove the lower bound inequality.
Assume that $v_k\in C^1(B_\radius;\R^2)$ is such that $v_k\rightarrow u$ strictly $BV(B_1;\R^2)$. Then
 for almost every $\rho<\radius$, 
there exists a subsequence $(v_{k_h})$ (depending on $\rho$) 
such that its restriction to $\partial B_\rho$ converges strictly $BV(\partial B_\rho;\R^2)$ to $u_{|\partial B_\rho}$. So, fix $\eps<1$ and a 
not-relabeled 
subsequence of $(v_k)$ such that
\begin{equation}\label{strict}
{v_k}_{|\partial B_\eps}\rightarrow u_{|\partial B_\eps}\quad \mbox{strictly }BV(\partial B_\eps;\R^2).
\end{equation}
Now, define $w_k:B_\radius\rightarrow\R^2$ as
\begin{equation}
w_k(x)=\left\{
\begin{aligned}
&v_k(x)\quad \mbox{if } |x|\leq\eps\\
&\frac{\radius-|x|}{\radius-\eps}v_k\left(\eps\frac{x}{|x|}\right)+\frac{|x|-\eps}{\radius-\eps}u\left(\eps\frac{x}{|x|}\right)\quad \mbox{if } \eps \leq |x|\leq \radius.
\end{aligned}
\right.
\end{equation}
Then $w_k$ is Lipschitz and $w=u$ on $\partial B_\radius$. Moreover, by \eqref{strict}, the convergence of $v_k$ to $u$ on $\partial B_\eps$ is also uniform, so we have (see the proof of \cite[Proposition 3.3, (3.29)]{BCS})
\begin{align}\label{jacob infinitesimal}
\lim_{k\rightarrow +\infty}\int_{B_\radius\setminus B_\eps}|Jw_k|dx=0.
\end{align}
Finally, since $w_k=v_k$ in $B_\eps$, by \eqref{jacob infinitesimal} we get
\begin{align}\label{lower limit}
\liminf_{k\rightarrow +\infty}\int_{ B_\radius}|Jv_k|dx\geq\liminf_{k\rightarrow +\infty}\int_{ B_\eps}|Jv_k|dx=\liminf_{k \rightarrow +\infty}\int_{B_\radius}|Jw_k|dx\geq P(u\mres\partial B_\radius)=P(\varphi_\radius)=P(\homogeneousmap),
\end{align}
where we used \eqref{resc}. We conclude by taking the infimum in the left hand side.
\end{proof}

\begin{Corollary}\label{area and Plateau}
Let $\varphi$ and $u$ as in Theorem \ref{TV and Plateau}. Then 
\begin{equation}\label{eq:area_for_homo}
\overline{\mathcal{A}}_{BV}(u;B_\radius)=\int_{B_\radius}\sqrt{1+|\nabla u|^2}dx+
P(\homogeneousmap).
\end{equation}
\end{Corollary}
\begin{proof}
For the lower bound, suppose that $v_k\in C^1(B_\radius;\R^2)$ is such that $v_k\rightarrow u$ strictly $BV(B_\radius;\R^2)$. Now, let $\eps<\radius$ such that \eqref{strict} holds, and write
${\mathcal A}(v_k;B_\radius) = {\mathcal A}(v_k;B_\radius \setminus B_{\eps})
+
{\mathcal A}(v_k;B_{\eps}) \geq 
{\mathcal A}(v_k;B_\radius \setminus B_{\eps})
+ \int_{B_{\eps}} \vert J v_k\vert dx$,
so that, by \cite[Theorem 3.7]{AD},
\begin{equation}\nonumber
\begin{aligned}
\lim_{k\rightarrow+\infty}
{\mathcal{A}(v_k;B_\radius)}&\geq\liminf_{k\rightarrow+\infty}\mathcal{A}
(v_{k};B_\radius\setminus B_{\varepsilon})+\liminf_{k\rightarrow+\infty}\int_{B_{\epsilon}}|Jv_{k}|dx
\\
&\geq\int_{B_\radius\setminus B_{\varepsilon}}\sqrt{1+|\grad u|^2}dx+\liminf_{k\rightarrow+\infty}\int_{B_{\epsilon}}|Jv_{k}|dx.
\end{aligned}
\end{equation}
We now apply \eqref{lower limit} and next pass to the limit as 
$\eps\rightarrow 0^+$ to get the lower bound in \eqref{eq:area_for_homo}.

Concerning the proof of the 
upper bound for \eqref{eq:area_for_homo}, consider the sequence $(v_k)$ 
defined in \eqref{rec},
which converges to $u$ in $W^{1,1}(B_\radius; \R^2)$. Then, upon extracting a subsequence such that $(\grad v_k)$ converges 
almost everywhere to $\grad u$, by \eqref{lim jac} and dominated convergence we have,
using the inequality $\sqrt{1 + a^2 + b^2  + c^2} \leq \sqrt{1 + a^2 + b^2} + \vert c\vert$ for
$a,b,c \in \R$,
\begin{equation}\nonumber
\begin{aligned}
\overline{\mathcal{A}}_{BV}(u;B_\radius)&\leq\limsup_{k\rightarrow + \infty}\mathcal{A}(v_k;B_\radius)\leq\lim_{k
\rightarrow +\infty}\int_{B_\radius}\sqrt{1+|\grad v_k|^2}dx+\lim_{k\rightarrow
+\infty}\int_{B_\radius}|Jv_k|dx\\
&=\int_{B_\radius}\sqrt{1+|\grad u|^2}dx+\int_{B_1}|Jv|dx,
\end{aligned}
\end{equation}
for any $v\in \mathrm{Lip}(B_1;\R^2)$ such that $v=\homogeneousmap$ on $\partial B_1$.
Passing to the infimum on the right hand side we 
obtain the upper bound inequality in \eqref{eq:area_for_homo}.
\end{proof}

Now, we treat the case $\gamma\in BV(\Suno;\R^2)$. We recall that, by Proposition \ref{homog extension is BV}, its homogeneouos extension $u$ is still $BV(B_\radius;\R^2)$. 
\begin{Theorem}\label{Plateau for BV}
Let $\gamma\in BV(\Suno;\R^2)$ and $u$ as in \eqref{hom map}. Let $\widetilde\gamma:\Suno\rightarrow\R^2$ be as in Lemma \ref{strict implies uniform 3}. Then
\begin{align}\label{TV for homo BV}
\overline{{\totvarjac}}_{BV}(u,B_\radius)= \rilP (\gamma)=P(\widetilde\gamma).
\end{align}
\end{Theorem}
\begin{proof}
In order to show the upper bound inequality, consider a Lipschitz 
sequence $\varphi_k:\Suno\rightarrow\R^2$  
converging to $\gamma$ strictly $BV(\Suno;\R^2)$ (e.g. a mollifying sequence). Then, by Lemma \ref{strict implies uniform 3}, there exists a equi-Lipschitz reparameterization $\widetilde\varphi_k$ of $\varphi_k$ 
that converges to $\widetilde\gamma$ uniformly (up to extracting a subsequence). Set 
\begin{align}\label{uk}
u_k(x)=\varphi_k\left(\frac{x}{|x|}\right)\quad\forall x\in B_\radius\setminus\{(0,0)\},
\end{align}
then $u_k\in W^{1,1}(B_\radius;\R^2)$ and $u_k\rightarrow u$ strictly $BV(B_\radius;\R^2)$, since
\begin{align*}
&\|u_k-u\|_{L^1(B_1;\R^2)}\leq\|\varphi_k-\gamma\|_{L^1(\Suno;\R^2)}
\rightarrow 0,
\\
& \int_{B_\radius}|\nabla u_k|dx=\radius\int_{\Suno}|\dot\varphi_k|d\mathcal{H}^1\rightarrow\radius|\dot\gamma|(\Suno)=|Du|(B_\radius),
\end{align*}
where we used Proposition \ref{homog extension is BV}.
Now, by lower semicontinuity of $\overline{{\totvarjac}}_{BV}(\cdot,B_\radius)$, Theorem \ref{TV and Plateau}, \eqref{inv}, and Lemma \ref{lem:continuity_of_P}, we have
$$
\overline{{\totvarjac}}_{BV}(u,B_\radius)\leq\liminf_{k\rightarrow +\infty}\overline{{\totvarjac}}_{BV}(u_k,B_\radius)=\liminf_{k\rightarrow +\infty}P(\varphi_k)=\liminf_{k\rightarrow +\infty}P(\widetilde\varphi_k)=P(\widetilde\gamma).
$$
Let us prove the lower bound inequality.
Assume that $v_k\in C^1(B_\radius;\R^2)$ is such that $v_k\rightarrow u$ strictly $BV(B_\radius;\R^2)$ and
$$
\lim_{k\rightarrow+\infty}\int_{B_\radius}|Jv_k|dx=\overline{{\totvarjac}}_{BV}(u,B_\radius).
$$
 We use Lemma \ref{lem:inheritance circ} to fix $\eps<\radius$ and a subsequence $(v_{k_j})\subset(v_k)$ such that $v_{k_j}\mres\partial B_\eps\rightarrow u\mres\partial B_\eps$ strictly $BV(\partial B_\eps;\R^2)$.
According to \eqref{def resc}, we have $u\mres\partial B_\eps=\gamma_\eps$. So, let $\widetilde\gamma_\eps$ be the Lipschitz curve of Lemma \ref{strict implies uniform 3} associated\footnote{We identify $\partial B_\eps$ with $[0,2\pi\eps]$.} to $\gamma_\eps$.
Using Corollary \ref{cor_continuity_P} and \eqref{resc}, we conclude
\begin{align}\label{lower bound ineq}
\overline{{\totvarjac}}_{BV}(u,B_\radius)&\geq\liminf_{j\rightarrow +\infty}\int_{B_\eps}|Jv_{k_j}|dx\geq\liminf_{j\rightarrow +\infty} P(v_{k_j}\mres\partial B_\eps)=\rilP(\gamma_\eps)=P(\widetilde\gamma_\eps)=P(\widetilde\gamma).
\end{align}
\end{proof}

\begin{Remark}\label{rem u tilde}
Setting $\widetilde{u}(x):=\widetilde{\gamma}\left(\frac{x}{|x|}\right)$, then $u\in W^{1,1}(B_\radius;\R^2)$. So, by Theorem \ref{TV and Plateau} and Theorem \ref{Plateau for BV}, we have
\begin{align}\label{u e u tilde}
\overline{{\totvarjac}}_{BV}(\widetilde{u},B_\radius)=\overline{{\totvarjac}}_{BV}(u,B_\radius).
\end{align}
\end{Remark}

We are in the position to prove Theorem \ref{relaxed area thm}.\\
{\it Proof of Theorem \ref{relaxed area thm}.}
For the lower bound, suppose that $v_k\in C^1(B_\radius;\R^2)$ is such that $v_k\rightarrow u$ strictly $BV(B_\radius;\R^2)$. Now, let $\eps<\radius$ such that \eqref{strict} holds, and write
${\mathcal A}(v_k;B_\radius) = {\mathcal A}(v_k;B_\radius \setminus B_{\eps})
+
{\mathcal A}(v_k;B_{\eps}) \geq 
{\mathcal A}(v_k;B_\radius \setminus B_{\eps})
+ \int_{B_{\eps}} \vert J v_k\vert dx$,
so that, by \cite[Theorem 3.7]{AD},
\begin{equation}\nonumber
\begin{aligned}
\lim_{k\rightarrow+\infty}
{\mathcal{A}(v_k;B_\radius)}&\geq\liminf_{k\rightarrow+\infty}\mathcal{A}
(v_{k};B_\radius\setminus B_{\varepsilon})+\liminf_{k\rightarrow+\infty}\int_{B_{\epsilon}}|Jv_{k}|dx
\\
&\geq\int_{B_\radius\setminus B_{\varepsilon}}\sqrt{1+|\grad u|^2}dx+|D^su|(B_\radius\setminus B_\eps)+\liminf_{k\rightarrow+\infty}\int_{B_{\epsilon}}|Jv_{k}|dx.
\end{aligned}
\end{equation}
We now apply \eqref{lower limit} and next pass to the limit as 
$\eps\rightarrow 0^+$ to get the lower bound in \eqref{area BV}.

Concerning the proof of the 
upper bound for \eqref{area BV}, consider the sequence $(u_k)\subset W^{1,1}(B_\radius;\R^2)$ 
defined in \eqref{uk},
which converges to $u$ strictly $BV(B_\radius; \R^2)$. 
Let us prove that 
\begin{align}\label{area strict}
\lim_{k\to+\infty}\int_{B_\radius}\sqrt{1+|\grad u_k|^2}dx=\int_{B_\radius}\sqrt{1+|\nabla u|^2}dx+|D^su|(B_\radius).
\end{align}
In polar coordinates, we get
$$
\int_{B_\radius}\sqrt{1+|\grad u_k|^2}dx=\int_0^\radius\int_0^{2\pi}\rho\sqrt{1+\frac{|\dot{\bar\varphi}_k(\theta)|^2}{\rho^2}}d\theta d\rho.
$$
For a fixed $\rho\in(0,\radius)$, consider $f_\rho(\xi)=\rho\sqrt{1+\frac{|\xi|^2}{\rho^2}}$, $\xi\in\R^2$. Then, $f_\rho$ is convex on $\R^2$. Now, if $\mu\in\mathcal{M}([0,2\pi];\R^2)$, one can consider the measure $f_\rho(\mu)\in\mathcal{M}^+([0,2\pi])$ defined as\footnote{See Theorem 2' in \cite{GS}: notice that $f^*_\rho=|\cdot|$ for every $\rho\in(0,\radius)$, where $f^*_\rho$ is the recession function associated to $f_\rho$.}
$$
f_\rho(\mu)(A)=\int_A\rho\sqrt{1+\frac{|a(\theta)|^2}{\rho^2}}d\theta+|\mu^s|(A),
$$
for any Borel set $A\subseteq[0,2\pi]$,
where $\mu^a=a\mathscr{L}^2$ for some $a\in L^1([0,2\pi])$. By \cite[Theorem 4]{GS}, $f_\rho(\cdot)$ is continuous w.r.t. the approximation by convolution. In particular, choosing $\mu:=\dot{\bar\gamma}\in\mathcal{M}([0,2\pi];\R^2)$ and $A=[0,2\pi]$, for every $\rho\in(0,\radius)$ we have
\begin{align*}
\lim_{k\to+\infty}f_\rho(\dot{\bar\varphi}_k)([0,2\pi])&=\lim_{k\to+\infty}\int_0^{2\pi}\rho\sqrt{1+\frac{|\dot{\bar\varphi}_k(\theta)|^2}{\rho^2}}d\theta\\
&=\int_0^{2\pi}\rho\sqrt{1+\frac{|\dot{\bar\gamma}^a(\theta)|^2}{\rho^2}}d\theta+|\dot{\gamma}^s|(\Suno)\\
&=f_\rho(\dot{\bar\gamma})([0,2\pi]).
\end{align*}
Integrating in $(0,\radius)$, by dominated convergence we infer
\begin{align*}
\lim_{k\to+\infty}\int_{B_\radius}\sqrt{1+|\grad u_k|^2}dx&=\lim_{k\to+\infty}\int_0^\radius\int_0^{2\pi}\rho\sqrt{1+\frac{|\dot{\bar\varphi}_k(\theta)|^2}{\rho^2}}d\theta d\rho\\
&=\int_0^\radius\int_0^{2\pi}\rho\sqrt{1+\frac{|\dot{\bar\gamma}^a(\theta)|^2}{\rho^2}}d\theta d\rho+\radius|\dot\gamma^s|(\Suno)\\
&=\int_{B_\radius}\sqrt{1+|\nabla u|^2}dx+|D^su|(B_\radius),
\end{align*}
where we used \eqref{nabla u} and \eqref{var split}. Therefore, we obtain \eqref{area strict}.\\
Finally, by lower semicontinuity of $\overline{\mathcal{A}}_{BV}(\cdot,B_\radius)$ and by Corollary \ref{area and Plateau}, we conclude
\begin{equation}\nonumber
\begin{aligned}
\overline{\mathcal{A}}_{BV}(u;B_\radius)&\leq\liminf_{k\rightarrow + \infty}\overline{\mathcal{A}}_{BV}(u_k;B_\radius)=\lim_{k\rightarrow + \infty}\left[\int_{B_\radius}\sqrt{1+|\grad u_k|^2}dx+P(\varphi_k)\right]\\
&=\int_{B_\radius}\sqrt{1+|\nabla u|^2}dx+|D^su|(B_\radius)+\rilP(\gamma).
\end{aligned}
\end{equation}
\qed

\begin{Remark}
We notice that, as a function of the set variable, $\overline{\totvarjac}_{BV}(u,\cdot)$ is a (finite) measure. Precisely, for every open set $A\subset B_\radius$
$$
\overline{\totvarjac}_{BV}(u;A)=\rilP(\gamma)\delta_0(A).
$$
Indeed, if $0\in A$ then $B_\eps\subset A$ for some $\eps\in(0,\radius)$ and we can argue as in \eqref{lower bound ineq}. On the other hand, suppose that $0\notin A$ and consider $u_k$ as in \eqref{uk}. Then, ${u_k}_{|A}\in$ Lip$(A;\R^2)$ and converges strictly $BV(A;\R^2)$ to $u_{|A}$. Since the image of $u_k$ has zero Lebesgue measure, by lower semicontinuity of $\overline{\totvarjac}_{BV}(\cdot\,;A)$, we get that $\overline{\totvarjac}_{BV}(u;A)=0$.\\
In the same way, one can prove that for every open set $A\subset B_\radius$
$$
\overline{\mathcal{A}}_{BV}(u;A)=\int_A\sqrt{1+|\nabla u|^2}dx+|D^su|(A)+\rilP(\gamma)\delta_0(A).
$$
Therefore, also $\overline{\mathcal{A}}_{BV}(u;\,\cdot)$ is a measure and \eqref{area BV} is an integral representation.
\end{Remark}

\begin{Remark}[{\bf On the Plateau problem \eqref{Plateau1}}]
Let $\varphi:\Suno\to\R^2$ be Lipschitz. From \cite[Theorem 1.3]{Cr}, there exists a least area mapping $v\in W^{1,p}(B_1;\R^2)$, for some $p>2$, spanning $\varphi$, i.e. realizing the infimum of the total variation of the Jacobian determinant in the class of Sobolev maps in $W^{1,p}(B_1;\R^2)$ whose trace on $\partial B_1$ is $\varphi$. In truth, one can prove that the least area mapping is Lipschitz, so that the Plateau problem \eqref{Plateau1} attains a minimum. The proof is a consequence of results contained in \cite{CrSt}: interestingly, it seems that one needs to pass through a more general metric result, concerning spaces with upper curvature bounds.
\end{Remark}

\textsc{Acknowledgements:}
The author is indebted to Giovanni Bellettini and Riccardo Scala for having suggested to write this paper and he thanks Paul Creutz for useful discussions.
The author is member of the Gruppo Nazionale per l'Analisi Matematica, la Probabilit\`a e le loro Applicazioni (GNAMPA) of the INdAM of Italy.


\end{document}